\documentclass[11pt]{amsart}

\usepackage{amssymb,amsmath,latexsym, amscd, graphicx, multirow, longtable, appendix}
\input cyracc.def

\newtheorem{theorem}{Theorem}[section]
\newtheorem{lemma}[theorem]{Lemma}
\newtheorem{proposition}[theorem]{Proposition}

\newtheorem{corollary}[theorem]{Corollary}

\newtheorem{example}[theorem]{Example}

\newtheorem{question}[theorem]{Question}
\def\a{\alpha}

\def\e{\epsilon}

\def\C{\mathbb{C}}

\def\Q{\mathbb{Q}}
\def\R{\mathbb{R}}
\def\Z{\mathbb{Z}}

\begin{document}

\title{Ordered groups, eigenvalues, knots, surgery and L-spaces}

\begin{abstract}  We establish a necessary condition that an automorphism of a nontrivial finitely generated bi-orderable group can preserve a
bi-ordering:  at least one of its eigenvalues, suitably defined, must be real and positive.  Applications are given to knot theory, spaces which fibre over the circle and to the Heegaard-Floer homology of surgery manifolds.  In particular, we show that if a nontrivial fibred knot has bi-orderable knot group, then its Alexander polynomial has a positive real root.  This implies that many specific knot groups are not bi-orderable.  We also show that if the group of a nontrivial knot is bi-orderable, surgery on the knot cannot produce an $L$-space, as defined by Ozsv\'ath and Szab\'o. 

\end{abstract}

\date{\today}

\author[Adam Clay and Dale Rolfsen]{Adam Clay and Dale Rolfsen}
\address{Department of Mathematics\\
University of British Columbia \\
Vancouver \\
BC Canada V6T 1Z2} \email{aclay@math.ubc.ca \\ rolfsen@math.ubc.ca}
\urladdr{http://www.math.ubc.ca/~aclay/ \\ http://www.math.ubc.ca/~rolfsen/} 
\maketitle

\section{Introduction}
Orderable groups have recently found interesting applications in topology, for example in the study of foliations and similar structures on $3$-dimensional manifolds \cite{calegari, calegari-dunfield, roberts}, the existence of mappings of nonzero degree \cite{BRW, rol}, in the theory of braids \cite{dehornoy,DDRW}, knot theory \cite{ito, malyut} and dynamics \cite{ghys, navas}.  There is also evidence that Heegaard-Floer homology is connected with orderability of the fundamental group of a closed 3-manifold \cite{OS, peters, watson}.  In this paper we provide further evidence.

It is known that all knot groups are left-orderable \cite{BRW, howie-short} and that some knot groups enjoy orderings which are invariant under multiplication on both sides.  In \cite{PR1} it was shown that if a fibred knot's Alexander polynomial has {\em all} roots real and positive, then the knot group will be bi-orderable.  One of the main results of the present article is a partial converse:

\begin{theorem}  Suppose that $K$ is a nontrivial fibred knot in $S^3$ and the group $\pi_1(S^3 \setminus K)$ is bi-orderable. Then the Alexander polynomial $\Delta_K(t)$ must have at least one root (actually two) which is real and positive. 
\end{theorem}  

That criterion establishes that many fibred knots' groups cannot be bi-ordered.  See the last section of this paper for examples.
A consequence of this pertains to Heegaard-Floer homology.  Ozsv\'ath and Szab\'o \cite{OS2} define an {\em L-space} to be a closed 3-manifold $M$ such that $H_1(M; \Q) = 0$ and its Heegaard-Floer homology 
$\widehat{HF}(M)$ is a free abelian group of rank equal to $|H_1(M; \Z)|$.  Lens spaces, and more generally 
3-manifolds with finite fundamental group are examples of $L$-spaces.  We are able to use their results and a theorem of Ni \cite{ni} to show the following.

\begin{theorem}\label{L-space}  Suppose $K$ is a nontrivial knot in $S^3$ and the knot group $\pi_1(S^3 \setminus K)$ is bi-orderable. 
Then surgery on $K$ cannot produce an $L$-space. 
\end{theorem}  

We also derive restrictions on fibred knots for which surgery produces a manifold with bi-orderable fundamental group.

\begin{theorem}\label{surgery}
Suppose $K$ is a nontrivial fibred knot in $S^3$ and nontrivial surgery on $K$ produces a 3-manifold $M$ whose fundamental group is bi-orderable.  Then the surgery must be longitudinal (that is, zero-framed) and  
$\Delta_K(t)$ must have a positive real root.  Moreover, $M$ fibres over $S^1$.
\end{theorem}

These are applications of the following, which is our main theorem.

\begin{theorem}  Suppose an automorphism of a nontrivial finitely generated bi-orderable group $G$ preserves a bi-ordering.  Then its induced automorphism on the rational vector space $H_1(G; \Q)$ must have at least one positive real eigenvalue.  
\end{theorem}

The paper is organized as follows.  Section 2 gives the definitions of orderable groups and their properties that we will need in the sequel.  In section 3 we discuss eigenvalues and prove the main theorem.  Applications to knot theory and fibred spaces are discussed in section 4, and in section 5 we give an application to Dehn surgery.  Section 6 proves Theorem \ref{L-space}, stated in equivalent form as Theorem \ref{L-surgery}. In the final section we list prime knots of 12 or fewer crossings whose groups are known to be bi-orderable (there are just twelve) as well as the 487 fibred knots whose groups are known {\em not} to be bi-orderable according to our criteria. 

{\bf Acknowledgements:}  We would like to thank Steve Boyer, Cameron Gordon, Peter Linnell, Dave Witte Morris and Andr\'es Navas for useful comments during the preparation of this paper.  Chuck Livingston's wonderful website {\em KnotInfo} (jointly maintained by Jae Choon Cha) was invaluable for our calculations, and we also thank Chuck for generously making the site's underlying database available to us.  We especially thank Liam Watson, whose enthusiastic interest and advice was crucial for our application to Heegaard-Floer homology.  The generous support of the Canadian Natural Sciences and Engineering Research Council is also gratefully acknowledged.

\section{Orderable groups} 

A group $G$ is {\em left-orderable} if there is a strict total ordering $<$ of its elements which is invariant under multiplication on the left: $g < h$ implies $fg < fh$ for $f,g,h \in G$.   It is easy to see that a group is left-orderable if and only if it is right-orderable.  An ordering of $G$ which is invariant under multiplication on both sides will be called a bi-ordering; if such an ordering exists we say that $G$ is {\em bi-orderable}.  Traditionally in the literature, such groups are called, simply, ``orderable,'' but we will use ``bi-ordering'' and ``bi-orderable'' to emphasize the two-sided invariance.
A mapping $\phi: G \to G'$ of ordered groups $(G, <)$ and $(G', <')$ is {\em order-preserving} if $g < h$ implies 
$\phi(f) <' \phi(g)$.  A subset $S$ of a left- or bi-ordered group $(G,<)$ is {\em convex} if $s_1 < g < s_2$ and $s_1, s_2 \in S, g \in G$ imply that $g \in S$.  If $C$ is a convex normal subgroup of $G$, an ordering of $G$ induces an ordering of $G/C$ by comparing representatives of cosets. This is well-defined and left- or bi-invariant if the ordering of $G$ is left- or bi-invariant.  Moreover, if $\phi: G \to G$ is an automorphism which preserves an order $<$ of $G$ and $\phi(C) = C$ for the convex normal subgroup $C$, then the induced map $\phi_C : G/C \to G/C$ preserves the induced order of $G/C$.

Given a left-ordering $<$ on $G$, the {\em positive cone} is defined to be the set $P = \{ g \in G | g>1 \}.$  It satisfies
(1) $P$ is closed under multiplication and (2) for every element $g \ne 1$ of $G$, exactly one of $g$ or $g^{-1}$ belongs to $P$.  Conversely, given a subset $P$ satisfying (1) and (2), a left-ordering of $G$ can be defined by declaring
$g < h$ if and only if $g^{-1}h \in P$.  A bi-ordering is characterized by having a positive cone satisfying (1), (2) and
(3) $g^{-1}Pg \subset P$ for every $g \in G$.

An ordering $<$ of a group $G$ is {\em Archimedian} if the powers of each of its nonidentity elements are cofinal in the ordering: if $g \ne 1$ and $h \in G$, there exists $n \in \Z$ so that $g^{-n} < h < g^n$.  Following is one of the early and basic theorems in orderable group theory \cite{holder} (see e.g. \cite{rhem} for a proof in English).

\begin{theorem}[H\"older, 1901]  If $G$ is a group with an Archimedian bi-ordering $<$, then there is an injective homomorphism of $G$ into the additive real numbers $\R$ which also preserves orderings (with the natural order on 
$\R$).  In particular, $G$ is abelian.  

\end{theorem}
  
Left-orderable groups are torsion-free, but not conversely, and have good algebraic properties.  For example, if $G$ is left-orderable, then its group ring $\Z G$ has no zero divisors -- a property conjectured to be true for all torsion-free groups.  For abelian groups, being left-orderable is of course equivalent to being bi-orderable, which is also equivalent to being torsion-free.

\begin{example}\label{rationalex}
Consider the group $G = \Q^n$, for which we will use additive notation.  Choose real numbers 
$\a_1, \dots, \a_n \in \R$ which are linearly independent when considering the real numbers $\R$ as a vector space over the rationals $\Q$.  Then there is an embedding $\phi$ of $G$ into $\R$ by the formula:
$$\phi (x_1, \dots, x_n) = x_1\a_1 + \cdots + x_n\a_n.$$
Define an ordering $<$ of $G$ by declaring $g<h$ if and only if $\phi(g) < \phi(h)$ in $\R$.  Since the natural ordering of $\R$ is Archimedian, this defines an Archimedian bi-ordering of $\Q^n$.  In fact, by H\"older's theorem, all Archimedian bi-orderings of $\Q^n$ are of this form.  The positive cone of this ordering may be regarded geometrically as all points of $\Q^n$ which are on one side of the hyperplane in $\R^n$ defined by the normal vector  $(\a_1, \dots, \a_n)$.  The only convex subgroups relative to this ordering are $\{0\}$ and $\Q^n$.  Also note that $\phi$ is continuous, relative to the usual topologies of $\Q^n$ and $\R$.

By contrast, we can define a non-Archimedian bi-ordering $\prec$ of $\Q^n$ using the usual ordering of $\Q$ and ordering vectors lexicographically: $$(x_1, \dots, x_n) \prec (y_1, \dots, y_n)$$ if for some $j$ we have $x_i = y_i$ for
$i<j$ and $x_j < y_j$.  Relative to this ordering there are $n+1$ distinct convex subgroups $C_j$, $j = 0, \dots, n$.  Here $C_0$ is the trivial subgroup, $C_n = \Q^n$ and for $0 < j < n$, $C_j$ is defined by the equation 
$x_1 = \cdots = x_{n-j} = 0$ and has dimension $j$. 

An easy argument shows that under {\em any} bi-ordering of $\Q^n$, a convex subgroup must actually be a vector subspace. 
\end{example}

Orderability of groups is obviously inherited by subgroups, but not necessarily by quotients (unless the kernel is convex in some ordering). 
The following are standard facts in ordered group theory, but we include proofs for the reader's convenience.  

\begin{lemma}\label{extension}
Suppose $H$ is a normal subgroup of $G$ and both $H$ and $G/H$ are left-orderable.  Then $G$ is left-orderable.
If both $H$ and $G/H$ are bi-orderable, then $G$ is bi-orderable if and only if some bi-ordering of $H$ is preserved under conjugation by all elements of $G$. 
\end{lemma}

\begin{proof} Define a positive cone $P$ for $G$ to be the union of the positive cone of $H$ and the preimage of the positive cone of $G/H$ under the projection.  It is routine to check that $P$ satisfies the conditions (1) and (2) above (and (3) in the bi-ordered case). If $G$ is bi-orderable, then the ordering restricted to $H$ is clearly preserved by conjugation.  
\end{proof}

Note that in this proof, $H$ is convex in the ordering constructed for $G$.

\begin{lemma} \label{ordlemma} 
If $G$ is a finitely-generated nontrivial bi-orderable group, then for any given bi-ordering $<$ there exists a unique maximal convex subgroup $C$ of $G$ satisfying $C \ne G$.  Moreover, $C$ is normal in $G$ and $G/C$ is abelian. 
If an automorphism $\phi : G \to G$ preserves $<$ then $\phi(C) = C.$  
\end{lemma}

\begin{proof}  Let $g_1, \dots , g_k$ be a generating set for $G$, with $k$ minimal.  We assume without loss of generality that for a given bi-ordering $<$ we have $1 < g_1< \cdots < g_k$.  Let $C$ be the union of all convex subgroups of $G$ which do {\em not} contain $g_k$.   Any union of convex subgroups is a convex subgroup, so $C$ is a convex subgroup and not equal to $G$, since it does not contain $g_k$.  Clearly any strictly larger convex subgroup will equal $G$.  Since, for any two convex subgroups, one must contain the other, uniqueness of $C$ is clear. If $\phi$ preserves $<$, then $\phi(C)$ is also a maximal proper subgroup, so $\phi(C) = C.$  In particular this applies to conjugation, so $C$ is normal.  The induced bi-ordering of $G/C$ is Archimedian (otherwise one would have a larger proper convex subgroup) so by 
H\"older's theorem $G/C$ is abelian.
\end{proof}

We will need to consider the tensor product of an abelian group $A$ with the rational numbers $\Q$ (considered as modules over $\Z$) and extend a bi-ordering of $A$ to a bi-ordering of $A \otimes \Q$.  Using additive notation,
given a bi-ordering $<$ of $A$, and an element $a_1 \otimes (p_1/q_1) + \cdots + a_n \otimes (p_n/q_n)$ of 
$A \otimes \Q$, we declare it to be in the positive cone of $A \otimes \Q$ if and only if
$$p_1(q/q_1) a_1 + \cdots  + p_n(q/q_n) a_n > 0$$ in $A$, where $q = |q_1\cdots q_n|$, so that the coefficients are integers.  The following is easily checked.

\begin{lemma} \label{tensorlemma}
Considering the abelian group $A$ as a subgroup of $A \otimes \Q$ via $a \to a \otimes 1$, any bi-ordering $<$ of $A$ extends to a bi-ordering of $A \otimes \Q$ by the recipe described above.  If $\phi : A \to A$ is an automorphism which preserves $<$, then $\phi \otimes id :  A \otimes \Q \to A \otimes \Q$ preserves the extended ordering.
\end{lemma}

\section{Eigenvalues} 

Let $G$ be a finitely-generated group and $\phi \colon G \to G$ an automorphism, or more generally, an endomorphism.   Consider the commutator subgroup $G'$ of $G$, and the induced automorphism $\phi_*$ of the abelianized group $G/G'$.   Tensoring with the rationals we obtain a
linear map of finite-dimensional vector spaces over $\Q$:
$$\phi_* \otimes id \colon G/G' \otimes \Q \to G/G' \otimes \Q.$$
The {\em eigenvalues} of $\phi$ are defined to be the the eigenvalues of this linear map, that is the roots of its characteristic polynomial $\chi_{\phi_* \otimes id}(\lambda) = \det (M - \lambda I)$, where $M$ is a matrix representing 
$\phi_* \otimes id$.  Note that, by the universal coefficient theorem, we have $G/G' \otimes \Q \cong H_1(G, \Q)$ and the
automorphism $\phi_* \otimes id$ is just the map induced by $\phi$ on rational homology.

Eigenvalues of automorphisms were considered in \cite{LRR}, where the following was proved.  Our main theorem will be a similar result for nonabelian groups.

\begin{proposition} \label{Pmain}
Let $A$ be a torsion-free abelian group of finite rank and let
$\theta$ be an automorphism of $A$.
Then $\theta$ preserves a bi-ordering if and only if for
each eigenvalue of $\theta$, at least one of its Galois conjugates
is a positive real number.
\end{proposition}

That paper also contains examples of two automorphisms of a free group, both of which have eigenvalues exactly the $n$-th roots of unity, for a given $n \ge 2$. However, one of the automorphisms preserves a bi-order, while the other, which is periodic of period $n$, cannot preserve a bi-order.  This shows in particular, that one cannot determine whether an automorphism of a finitely generated free group preserves a bi-order by looking at its action on 
the abelianization of the free group.  On the other hand, we will see that if an automorphism has {\em no} positive real eigenvalues it cannot preserve a bi-ordering.
For the reader's convenience, we offer a special case of Proposition \ref{Pmain} sufficient for our purposes, and a different, topological, proof.

\begin{proposition} \label{matrix}
Consider the rational vector space $\Q^n$, where $n$ is a positive integer, and let $M : \Q^n \to \Q^n$ be an automorphism represented by the nonsingular matrix $M$.  If $M$ preserves a bi-order of $\Q^n$, considered as an
additive abelian group, then $M$ has a positive real eigenvalue.
\end{proposition}

\begin{proof}  The case $n=1$ is rather trivial, so we assume $n \ge 2$.
With the natural inclusion $\Q^n \subset \R^n$, $M$ also represents an automorphism of $\R^n$.  Let $H \subset \R^n$ be the set of all $x \in \R^n$ such that every neighbourhood of $x$ contains points which are positive and points which are negative in the given bi-order of $\Q^n$.  By Lemma \ref{subspace} below, $H$ is a subspace of $\R^n$ of dimension $n-1$, and its complement is the union of two disjoint open sets $U_+$ and $U_-$, which intersect 
$\Q^n$ in points which are, respectively, positive and negative in the given ordering.  Also, $H$ is invariant under the action of $M$.

$H$ intersects the unit sphere $S^{n-1}$ of $\R^n$ in an $(n-2)$-sphere, which separates $S^{n-1}$ into two open disks of dimension $n-1$, one of which lies in $U_-$ and the other in $U_+$.  Let $D_+$ denote the closure of the latter disk.
Now $M$ induces a continuous function $\hat{M} : S^{n-1} \to S^{n-1}$ by the formula $$\hat{M} (x) = M(x)/|M(x)|.$$
As $M$ preserves the given ordering of $\Q^n$ it also sends the closure of $U_+$ to itself, and therefore $\hat{M}$ maps
$D_+$ to itself.  By a well-known theorem of Brouwer, $\hat{M}$ has a fixed point in $D_+$, which corresponds to an eigenvector of $M$ with positive eigenvalue. 
\end{proof}

\begin{lemma}\label{subspace}
As defined above, $H$ is a subspace of dimension $n-1$.
\end{lemma}

\begin{proof}
Let $\Q^n_+$ and $\Q^n_-$ denote the points of $\Q^n$ which are greater (resp. less) than 0 in the given
bi-ordering of $\Q^n$.  Then 
$$H = \{ x \in \R^n | \forall \e > 0, \exists x_+ \in \Q^n_+ ,  x_-\in \Q^n_- \; with \; |x - x_+| < \e \; and \; |x - x_-| < \e\}$$

We first check that $H$ is a linear subspace:\newline 
(i) $x, y \in H \Rightarrow x+y \in H$:  Let $\e > 0$ be given and choose $x_+$ and $y_+$ in $\Q^n_+$ so that $|x - x_+| < \e/2$ and 
$|y - y_+| < \e/2$.  Then $x_+ + y_+ \in \Q^n_+$ and 
$$|(x+y) - (x_+ + y_+)| \le |x - x_+| + |y - y_+| < \e.$$
Similarly, the $\e$ neighbourhood of $x+y$ contains points of $Q^n_-$. \newline
(ii) $x \in H \Rightarrow -x \in H$:  This is left to the reader. \newline
(iii) $x \in H, \a \in \R \Rightarrow \a x \in H:$  By the above, we may assume $\a > 0$.  Given $\e >0$, 
choose $x_+ \in \Q^n_+$ so that $|x - x_+| < \e/{2\alpha}$.  Then choose $r \in \Q$ so that $|\a - r| < \e/{2|x_+|}$ and $r>0$.
Then $r x_+$ belongs to $\Q^n_+$ and
$$|\a x - r x_+| \le |\a x - \a x_+| + |\a x_+ - r x_+| < \e/2 + \e/2 = \e.$$
Similarly for points of $\Q^n_-$.

We've established that $H$ is a linear subspace of $\R^n$; it remains to establish its dimension.  The complement of $H$ is the union of two disjoint open sets: 
the set $U_+ $ of all $x \in \R^n$ such that some neighbourhood of $x$ intersects $\Q^n$ only in points of $\Q^n_+$, and 
$U_-$ defined similarly, for points of $\Q^n_-$.  We must verify that, in fact, $U_+$ and $U_-$ are nonempty.  To that end, referring to Lemma \ref{ordlemma} and Example \ref{rationalex}, consider the maximal proper convex subgroup 
$C \subset \Q^n$, which is a subspace, the projection $\pi: \Q^n \to \Q^n/C$ and the embedding $\phi: \Q^n/C \to \R$ given by H\"older's theorem.  Both maps are continuous, and $(\phi\pi)^{-1}(0, \infty)$ is a nonempty open subset of 
$U_+$.  Similarly, $U_-$ is nonempty.
Since $H$ separates $\R^n$ into two components, its dimension must be $n-1$.
\end{proof}

We can now prove the main result of this paper.

\begin{theorem}\label{poseval}
Suppose $G$ is a nontrivial finitely generated bi-orderable group and 
that $\phi : G \to G$ preserves a bi-ordering of $G$.  Then $\phi$ has a positive eigenvalue.
\end{theorem}

\begin{proof}
If we knew that the commutator subgroup $G'$ were convex in the ordering, then we would be done, as we would have an induced bi-ordering of $G/G'$ preserved by the induced map $\phi_* : G/G' \to G/G'$, and Proposition \ref{matrix} would
apply to $\phi_* \otimes id :   G/G' \otimes \Q \to  G/G' \otimes \Q$.  But unfortunately $G'$ need not be convex, so we use a somewhat different approach.
Applying Lemma \ref{ordlemma}, let $C$ be the maximal proper convex subgroup of $G$, with respect to the 
bi-ordering $<$ preserved by $\phi$.
Since the quotient $G/C$ is abelian,  $G'$ is contained in $C$ and we have a short exact sequence 
\[
0 \rightarrow C/G' \rightarrow G/G' \rightarrow G/C \rightarrow 0 
\]
of abelian groups.  Moreover since $\phi(C) = C$, we have induced maps 
\[
\begin{CD}0 @>>> C/G' @>>> G/G' @>>> G/C @>>> 0\mbox{ } \\
 & & @VVV @V{\phi_*}VV @V{\phi_C}VV \\
0 @>>> C/G' @>>> G/G' @>>> G/C @>>>0.
\end{CD}
\]
Note that the map $\phi_C$ preserves the bi-ordering of $G/C$ induced from the bi-ordering of $G$.

Writing $U = C/G' \otimes \Q$, $V= G/G' \otimes \Q$, and $W= G/C \otimes \Q$, tensoring with $\Q$ (which is a right-exact functor) yields the commutative diagram of finite-dimensional vector spaces over $\Q$ with exact rows:
\[
\begin{CD}  U @>>> V @>>> W @>>> 0\mbox{ }  \\
  @VVV @V{\phi_V}VV @V{\phi_W}VV \\
 U @>>> V @>>> W @>>>0 ,
\end{CD}
\]
where $\phi_W = \phi_C \otimes id$ and $\phi_V = \phi_* \otimes id$.
Let $K$ denote the kernel of the quotient map $V \rightarrow W$, and denote by $\phi_K$ the restriction of $\phi_V$ to the subspace $K$, which is easily checked to be invariant under $\phi_V$.  Since every exact sequence of vector spaces splits we may write $V= K \oplus W$ and $\phi_V = \phi_{K} \oplus \phi_W$. Therefore the characteristic polynomial of $\phi_V$ factors as
\[ \chi_{\phi_V}(\lambda) = \chi_{\phi_{K}}(\lambda) \cdot \chi_{\phi_W}(\lambda) .
\]

By Lemmas \ref{ordlemma} and \ref{tensorlemma}, $\phi_W$ preserves the bi-ordering on $W$ induced by that of $G$.  Now choose a basis for $W$ and apply
Proposition \ref{matrix} to conclude that $\chi_{\phi_W}(\lambda)$ has a positive real root.  It follows that this is also a root of $\chi_{\phi_V}(\lambda)$, and therefore an eigenvalue of $\phi$.
\end{proof}

The automorphism $\phi$ may also have eigenvalues which are not real.  For instance, in the example of \cite{LRR} mentioned above, the automorphism of the free group preserves a bi-order and has eigenvalues exactly the $n$-th roots of unity.  

Dave Witte Morris has pointed out that Theorem \ref{poseval} can be strengthened by replacing ``bi-ordering" with the more general ``Conrad ordering" \cite{conrad}, such that to every pair $g, h$ of positive elements there is a positive integer $n$ such that $(gh)^n>hg$, using essentially the same proof.  We will not need this more general version here.

\section{Fibrations and fibred knots}

Suppose $X$ is a topological space and $f : X \to X$ a continuous function.  Then the mapping torus $M_f$ is defined by 
$$M_f  \cong X \times [0,1] / \sim $$
where one makes the identifications $(x,1) \sim (f(x),0)$.  An important instance is any space which is a (locally trivial) fibration over the circle $S^1$, in which the total space can be regarded as the mapping torus of a homeomorphism $f$ of the fibre.  The map $f$ is called the (topological) monodromy associated with the fibration, and is defined up to isotopy.

The fundamental group of a mapping torus is an HNN extension of $\pi_1(X)$.  If $f_\sharp : \pi_1(X) \to \pi_1(X)$ is
the induced map (homotopy monodromy), then
$$ \pi_1(M_f) \cong \langle \pi_1(X), t \; | \; t^{-1}gt = f_\sharp(g),  \forall g \in \pi_1(X) \rangle .$$

We have an exact sequence
$$1 \rightarrow \pi_1(X) \rightarrow \pi_1(M_f) \rightarrow \Z \rightarrow 1.$$

Lemma \ref{extension} implies the following:

\begin{proposition}
The fundamental group of $M_f$ is left-orderable if and only if $\pi_1(X)$ is left-orderable. 
\end{proposition}

\begin{proposition}\label{HNN}
The fundamental group of $M_f$ is bi-orderable if and only if $\pi_1(X)$ admits a bi-ordering which is preserved by 
$f_\sharp$.
\end{proposition}

\begin{proposition}
If $Y \to S^1$ is a fibration with fibre $F$, then $\pi_1(Y)$ is bi-orderable if and only if $\pi_1(F)$ is bi-orderable and the homotopy monodromy  $f_\sharp$ preserves a bi-ordering of $\pi_1(F)$.
\end{proposition}

From Theorem \ref{poseval} we have the following.

\begin{corollary}\label{monodromy}
Suppose $Y$ fibres over $S^1$ with fibre $F$, $ \pi_1(Y)$ is bi-orderable and $\pi_1(F)$ is nontrivial and finitely generated.  Then the monodromy $f_\sharp$ must have a real positive eigenvalue.
\end{corollary}

This can be considered a partial converse to results of \cite{PR1} and \cite{PR2} which we state as follows:

\begin{theorem}\label{perron}
Suppose $Y$ fibres over $S^1$ with fibre $F$.  Suppose $\pi_1(F)$ is a finitely generated free group or the fundamental group of a compact orientable surface and that all eigenvalues of the homotopy monodromy $f_\sharp$ are real and positive.  Then $\pi_1(Y)$ is bi-orderable.  
\end{theorem}

Recall that a knot $K$ in $S^3$, or more generally in a closed orientable 3-manifold $M$ is {\em fibred} if $M \setminus K$ fibres over $S^1$ with fibres being open surfaces.  It is well-known that for a fibred knot $K$ in $S^3$ or, more generally, in a homology sphere, the Alexander polynomial $\Delta_K(t)$ is the characteristic polynomial of the associated (homotopy) monodromy.  This implies the following.

\begin{theorem}\label{fibknot}
Suppose $K$ is a nontrivial fibred knot in a homology sphere and its knot group $\pi_1(M \setminus K)$ is bi-orderable.  Then $\Delta_K(t)$ has a positive real root.
\end{theorem}

Theorem \ref{perron}, on the other hand, implies that if {\em all} roots of a fibred knot's Alexander polynomial are real and positive, then its knot group is bi-orderable.

In fact, by the well-known symmetry condition of Alexander polynomials $\Delta_K(t)$ will have two positive real roots (mutual reciprocals) if it has one.  Since $\Delta_K(1) = \pm 1$, $t=1$ is never a root.
Among nontrivial prime knots of up to eight crossings, we have the following data, compiled with the assistance of the website {\em knotinfo}.

\noindent {\bf Examples:}
Among prime knots of up to eight crossings, the unknot and the fibred knots $4_1$ and $8_{12}$ have bi-orderable knot groups.  Their Alexander polynomials are:

\medskip

$\Delta_{4_1} = 1-3t+t^2$ with roots ${(3 \pm \sqrt{5})}/2.$

\medskip

$\Delta_{8_{12}} = 1-7t+13t^2-7t^3+t^4$ with roots (rounded to five decimals) 

$0.22778, 0.54411, 1.83785$ and $4.39026$ 

\medskip

Among the other seventeen fibred knots with at most eight crossings, the following have Alexander polynomials with no real roots, and therefore their knot groups are {\em not} bi-orderable: $3_1, 5_1, 6_3,  7_1, 7_7, 8_7, 8_{10}, 8_{16}, 8_{19}$ and $8_{20}$. 

\medskip 

For example $\Delta_{8_{19}} = 1-t+t^3-t^5+t^6 = (t^2 +\sqrt{3} t + 1)(t^2 -\sqrt{3} t + 1)(t^2 - t + 1)$.  
Its six roots are $(\sqrt{3} \pm i)/2, (-\sqrt{3} \pm i)/2$ and $(1 \pm i\sqrt{3} )/2$

\medskip

The polynomials of the remaining seven fibred knots up to eight crossings have both real and complex roots; we do not know if their groups are bi-orderable.  They are: $6_2, 7_6, 8_2, 8_5, 8_9, 8_{17}$ and $8_{18}$.

\medskip

The torus knots of type $(p, q)$, where $p$ and $q$ are coprime integers greater than one, have Alexander polynomial 
$$\Delta (t) = \frac{(1 - t)(1- t^{pq})}{(1 - t^p)(1- t^q)}$$
whose roots are all in $\C \setminus \R$, on the unit circle.  The groups of torus knots are therefore not bi-orderable, a fact which was already known, as they are fibred and their monodromy is periodic.

\medskip

The ``Fintushel-Stern" knot \cite{FS}, also known as the pretzel knot of type $(-2, 3, 7)$, is well-known for several reasons, such as admitting seven exceptional surgeries (see Cameron Gordon's discussion in Problem 1.77 of  Kirby's 
problem list \cite{kirby}).  It is fibred and its Alexander polynomial is $L(-t)$, where $L$ is the Lehmer polynomial
$$L(x) = 1 + x - x^3 - x^4 - x^5 - x^6 - x^7 + x^9 + x^{10}$$
which is the integral polynomial of smallest known Mahler measure \cite{boyd, hironaka}.  Among the ten roots 
of $L$, two
are real, the Salem number $1.17628 . . .$ and its reciprocal, and eight are complex, lying on the unit circle.  Therefore
the Alexander polynomial of the $(-2, 3, 7)$ pretzel knot has exactly two real roots and they are {\em negative.}  It follows that its knot group is not bi-orderable.  (We thank Liam Watson for pointing out this example.) 

Experimental evidence (see the final section of this paper) indicates that bi-orderability of the groups of approximately one-third the fibred knots can be decided by Theorems \ref{perron} and \ref{fibknot}.  

\section{Surgery}

We recall the definition of surgery on a knot $K$ in a 3-manifold $M$.  Consider a closed regular neighbourhood $N$ of $K$, so that $N \cong D^2 \times S^1$ and $\partial N \cong S^1 \times S^1$.  If $J$ is a homologically nontrivial simple closed curve on $\partial N$, it defines the surgery manifold:
$$M(K,J) = (M \setminus int N) \cup D^2 \times S^1 / \sim$$
where the boundaries are identified in such a way that $J$ is sewn to the meridian $\partial D^2 \times \{*\}$ of $D^2 \times S^1$.

This manifold is well-defined by the isotopy class of $J$ in $\partial N$, which in turn is determined by its homology class, up to sign, 
$\pm [J] = p\mu + q\lambda$, where $p$ and $ q$ are relatively prime integers (possibly the pair $\{0,1\}$) and 
$\mu, \lambda$ are a basis for the homology $H_1(\partial N) \cong \Z \oplus \Z$.  By convention, we take $\mu$ to correspond to the meridian which bounds a disk in $N$, with some chosen orientation, and $\lambda$ to be isotopic to $K$ in $N$.  There is a ``framing'' ambiguity for $\lambda$ unless $M$ is a homology sphere (that is, its homology groups coincide with those of $S^3$), in which case it corresponds to a curve on $\partial N$ which is homologically trivial in $M \setminus int N$.  With this convention one also refers to $M(K,J)$ as the result of $p/q$ surgery on $K$, where $p/q \in \Q \cup \infty$.  The ``trivial'' surgery corresponds to $\infty$, in which case the surgery manifold is just $M$ itself.  See \cite{KL} for further information on surgery.

\begin{theorem}\label{surgery}
Suppose $K$ is a fibred knot in $S^3$ and nontrivial surgery on $K$ produces a 3-manifold $M$ whose fundamental group is bi-orderable.  Then the surgery must be longitudinal (that is, 0-framed) and  
$\Delta_K(t)$ has a positive real root.  Moreover, $M$ fibres over $S^1$.
\end{theorem}

\begin{proof}
Let $M$ denote the result of $p/q$ surgery on $K$ in $S^3$.  We calculate, by a Meyer-Vietoris argument that $H_1(M) \cong \Z/p\Z$, which is a finite group unless $p = 0$.  By property P, $\pi_1(M)$ is nontrivial, so by Lemma \ref{ordlemma} its bi-orderability implies that $H_1(M)$ must be infinite, so we conclude that
$p=0$; the surgery curve $J$ is a longitude.  The neighborhood $N$ of $K$ for defining the surgery may be chosen so that $S^3 \setminus int N$ fibres over $S^1$ with each fibre an oriented surface $\Sigma$ with
boundary $\partial \Sigma$ a longitude parallel to $J$ on $\partial N$.  Now the solid torus $D^2 \times S^1$ also
clearly fibres over $S^1$ via projection onto the second coordinate, with fibre $D^2$.  Thus $M$ may be fibred by
matching these two fibrations, and the fibre is the closed surface $\hat{\Sigma}$ obtained by sewing a disk to $\Sigma$ along their boundaries.  The abelianizations of $\pi_1(\Sigma)$ and $\pi_1(\hat{\Sigma})$ coincide, and we see that the
monodromies of the fibrations of $M$ and $S^3 \setminus K$ also coincide upon abelianization, and so have the same characteristic polynomial, which is $\Delta_K(t)$.  By Corollary \ref{monodromy}, this polynomial has a positive real root.
\end{proof}

The following has almost the same proof.

\begin{theorem}
Suppose $K$ is a fibred knot in a homology sphere and nontrivial surgery on $K$ produces a 3-manifold $M$ whose fundamental group is nontrivial and bi-orderable.  Then the surgery must be longitudinal (that is, 0-framed) and  
$\Delta_K(t)$ has a positive real root.  Moreover, $M$ fibres over $S^1$.
\end{theorem}

Similar considerations hold for higher dimensional fibred knots, although in that case the fibre may not necessarily have bi-orderable (or even torsion free) fundamental group.

\begin{question}
Is there a version of this theory for non-fibred knots?
\end{question}

\begin{question}
What about fibred knots which have some, but not all, roots of their Alexander polynomial real and positive?
\end{question}

\begin{question}
If a knot in $S^3$ has bi-orderable group, does every manifold resulting from surgery on that knot have left-orderable fundamental group?
\end{question}

\section{$L$-spaces}

This section is devoted to the proof of the following theorem.

\begin{theorem}
\label{L-surgery}
If surgery on a knot $K$ in $S^3$ results in an $L$-space, then the knot group $\pi_1(S^3 \setminus K)$ is not bi-orderable.

\end{theorem}

As noted in \cite{ni}, if surgery on a knot yields an $L$-space, then the knot is fibred and also an {\em integer} surgery on the knot produces an $L$-space.  
From \cite{OS2}, we have the following theorem.

\begin{theorem}
\label{th:OS}
Let $K$ be a knot in $S^3$ for which integer surgery on $K$ yields an $L$-space.  Then the Alexander polynomial of $K$ has the form 
\[ \Delta_K(t) = (-1)^k+ \sum_{j=1}^k(-1)^{k-j}(t^{n_j} + t^{-n_j})
\]
for some increasing sequence of positive integers $0 < n_1 < n_2 < \cdots < n_k$.
\end{theorem}

We will show that such an Alexander polynomial can never have a positive real root.

\begin{lemma}
\label{lem:ineq}
Suppose that $\alpha > 1$ and $s >t >0$.  Then $\alpha^s + \alpha^{-s} > \alpha^t + \alpha^{-t}$.
\end{lemma}
\begin{proof}
For $\alpha > 1$, consider the function $f(x) = \alpha^x + \alpha^{-x}$.  It is continuous and differentiable for all $x$, with derivative $f'(x) = \ln(\alpha)(\alpha^x - \alpha^{-x})$.  Since $\alpha >1$, both $\ln(\alpha)$ and $\alpha^x - \alpha^{-x}$ are positive whenever $x >0$, hence $f'(x) >0$ for all $x>0$, and so $f$ is an increasing function on $(0, \infty)$.   Therefore, $s>t>0$ implies $f(s) >f(t)$, in other words  $\alpha^s + \alpha^{-s} > \alpha^t + \alpha^{-t}$.
\end{proof}

\begin{proposition}
\label{prop:norr}
Let $0 < n_1 < n_2 < \cdots < n_k$ be an increasing sequence of positive integers. Then the polynomial 
\[ \Delta_K(t) = (-1)^k+ \sum_{j=1}^k(-1)^{k-j}(t^{n_j} + t^{-n_j})
\]
 does not have a positive real root.
\end{proposition}
\begin{proof}
Suppose $\alpha$ is a positive real root of $\Delta_K(t)$.  We observe that $\alpha = 1$ is not possible, because 
\[ \Delta_K(1) = (-1)^k+ 2\sum_{j=1}^k(-1)^{k-j}, 
\]
which we may rewrite as $\Delta_K(1) = -1+ 2 \cdot 1 =1 $ if $k$ is odd, and  $\Delta_K(1) = 1+ 2 \cdot 0 =1 $ if $k$ is even.  Moreover, as $\Delta_K(t)$ is symmetric, $\alpha^{-1}$ must also be a root of $\Delta_K(t)$,  and so we may assume without loss of generality that $\alpha>1$. 

First, we consider the case when $k$ is odd, so that the quantity $\Delta_K(\alpha)$ can be written as
\[ -1 + (\alpha^{n_1}+\alpha^{-n_1}) - (\alpha^{n_2}+\alpha^{-n_2}) + \cdots +(\alpha^{n_k}+\alpha^{-n_k}).
\]
Observe that $\alpha^{n_1}+\alpha^{-n_1} -1 >0$, since $\alpha >1$ and $n_1 >0$.  For integers $i$ satisfying $1<i<k$, since $\alpha>1$ and $n_{i+1}>n_i$ we may apply Lemma \ref{lem:ineq} to conclude $(\alpha^{n_{i+1}}+\alpha^{-n_{i+1}}) - (\alpha^{n_{i}}+\alpha^{-n_{i}}) >0$.  Therefore the quantity $\Delta_K(\alpha)$ can be written as
\[ -1 + (\alpha^{n_1}+\alpha^{-n_1}) + \sum_{\substack{i=2\\i \text{ even}}}^{k-1} \left[ (\alpha^{n_{i+1}}+\alpha^{-n_{i+1}}) - (\alpha^{n_{i}}+\alpha^{-n_{i}}) \right] ,
\]
which is a sum of positive terms and so cannot be zero.

In the case that $k$ is even, we may write the quantity $\Delta_K(\alpha)$ as
\[ 1 - (\alpha^{n_1}+\alpha^{-n_1}) + (\alpha^{n_2}+\alpha^{-n_2}) - \cdots +(\alpha^{n_k}+\alpha^{-n_k}),
\]
which we may rewrite as
\[ 1 + \sum_{\substack{i=1\\i \text{ odd}}}^{k-1}  \left[ (\alpha^{n_{i+1}}+\alpha^{-n_{i+1}}) - (\alpha^{n_{i}}+\alpha^{-n_{i}}) \right].
\]
As in the case when $k$ is odd, we apply Lemma \ref{lem:ineq} to conclude that this is a sum of positive quantities, and so cannot be zero.
\end{proof}

This completes the proof of Theorem \ref{L-surgery}, since by the above 
the Alexander polynomial of a knot which produces an $L$-space cannot have a positive real root, and so the knot group is not bi-orderable.

\section{Examples up to 12 crossings}

Of the nontrivial prime knots with $12$ or fewer crossings, $1246$ of them are fibred, according to {\em knotinfo}, and so we may apply Theorems \ref{perron} and \ref{fibknot}. Among these fibred knots, we find that $487$ of them have non bi-orderable groups because their Alexander polynomials have no roots in $\R_+$, while $12$ have bi-orderable groups because all the roots are real and positive.  The bi-orderability of the remaining ones, whose Alexander polynomials have some, but not all, roots real and positive, is not known to us.

The table below contains all nontrivial prime knots with $12$ or fewer crossings whose groups are known to be bi-orderable.  The diagrams were produced using Rob Scharein's program {\em Knotplot}.

{
\renewcommand{\arraystretch}{2.2}
\begin{longtable}{ccc}
	\hline
\nopagebreak Knot&  & Alexander polynomial \\
  \hline

  \multirow{3}{*}{\includegraphics[height=20mm]{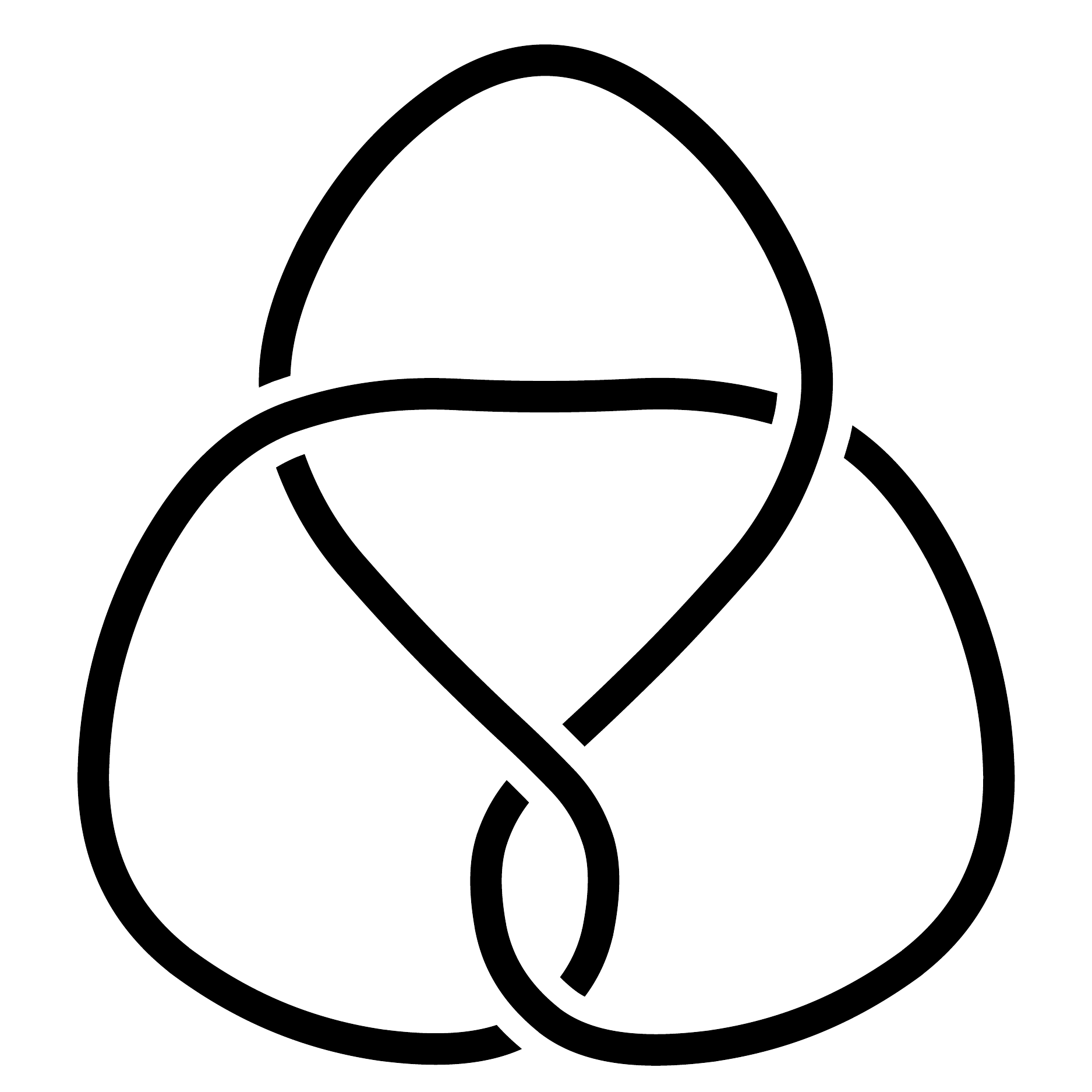}}&     \\
\nopagebreak   & $4_1$ & $1-3t+t^2$ \\
\nopagebreak  &  &  \\
  \multirow{3}{*}{\includegraphics[height=20mm]{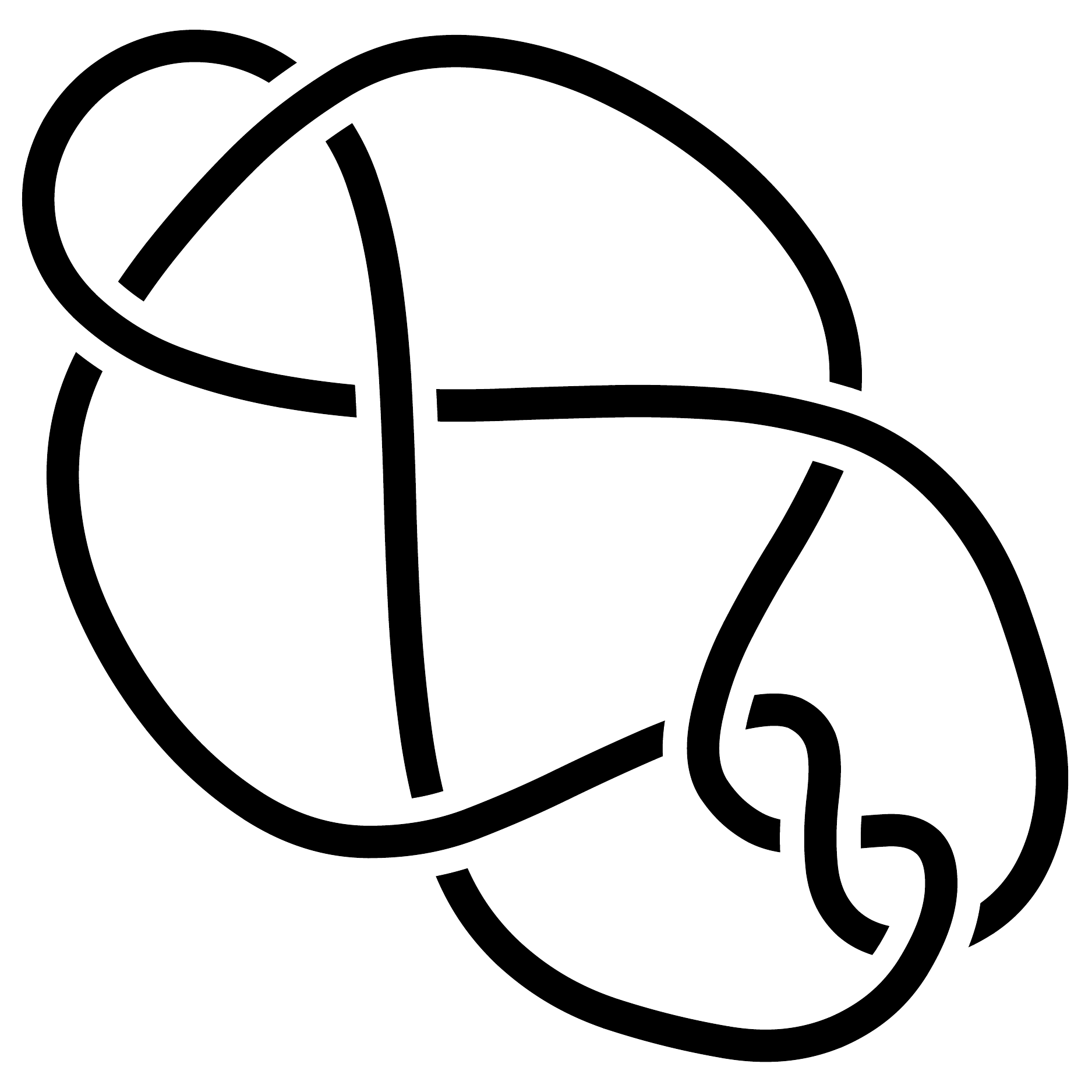}} &   \\
\nopagebreak   & $8_{12}$ & $1-7t+13t^2-7t^3+t^4$  \\
\nopagebreak  &  &  \\
  \multirow{3}{*}{\includegraphics[height=20mm]{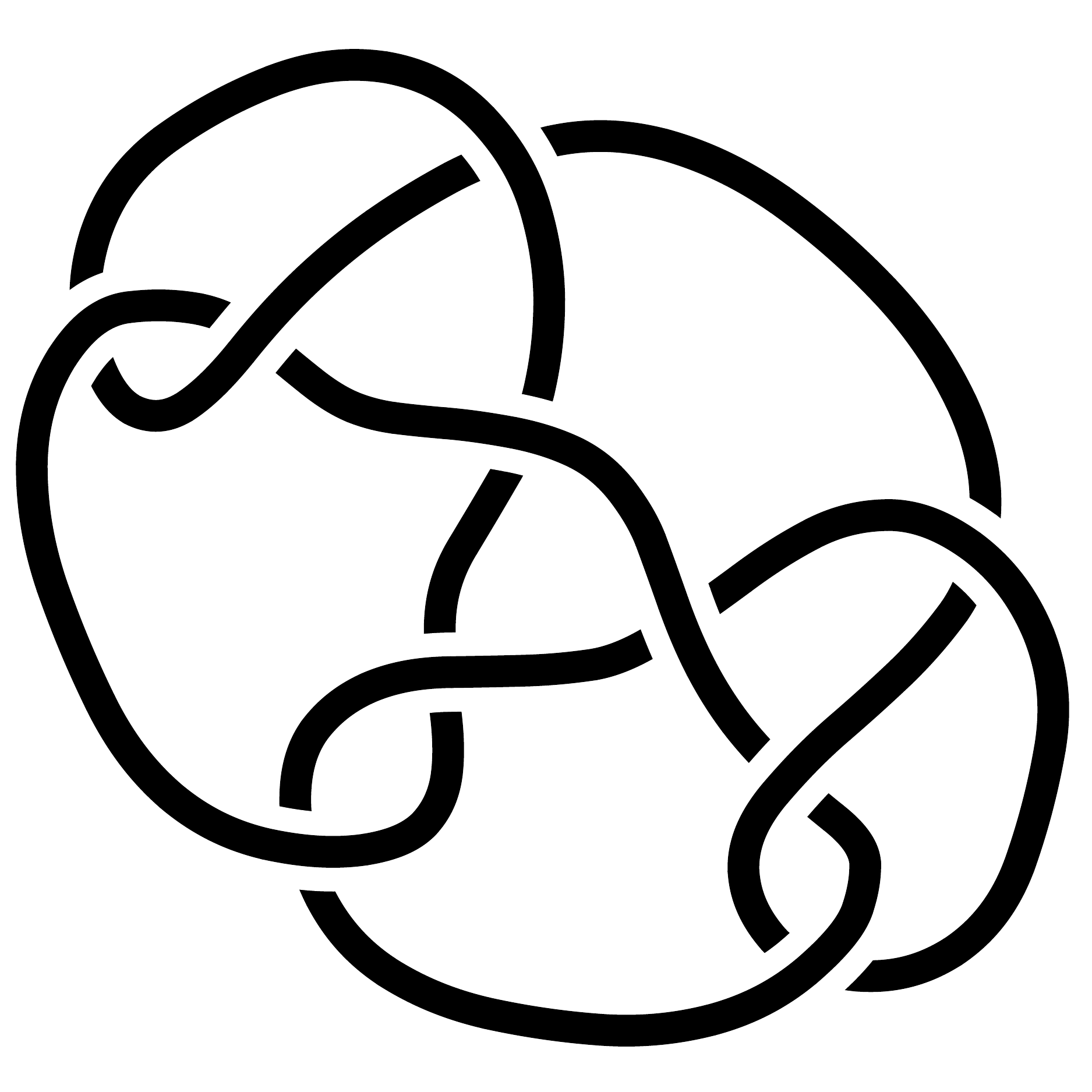}}&   \\
\nopagebreak   & $10_{137}$ & $1-6t+11t^2-6t^3+t^4$ \\
\nopagebreak  &  &  \\
  \multirow{3}{*}{\includegraphics[height=20mm]{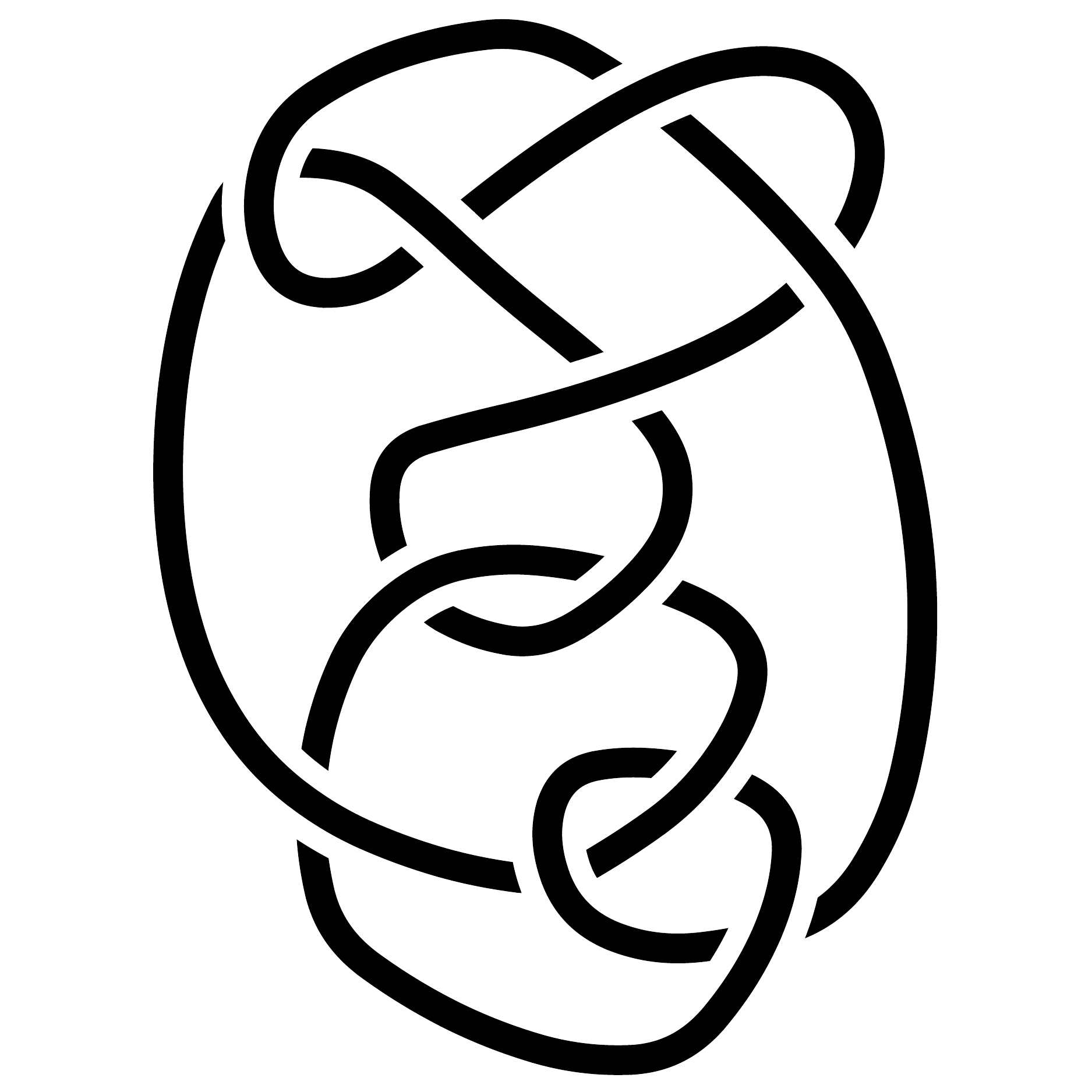}}&    \\
\nopagebreak   & $11a_{5}$ & $1-9t+30t^2-45t^3+30t^4-9t^5+t^6$  \\
\nopagebreak  &  & \\
\multirow{3}{*}{\includegraphics[height=20mm]{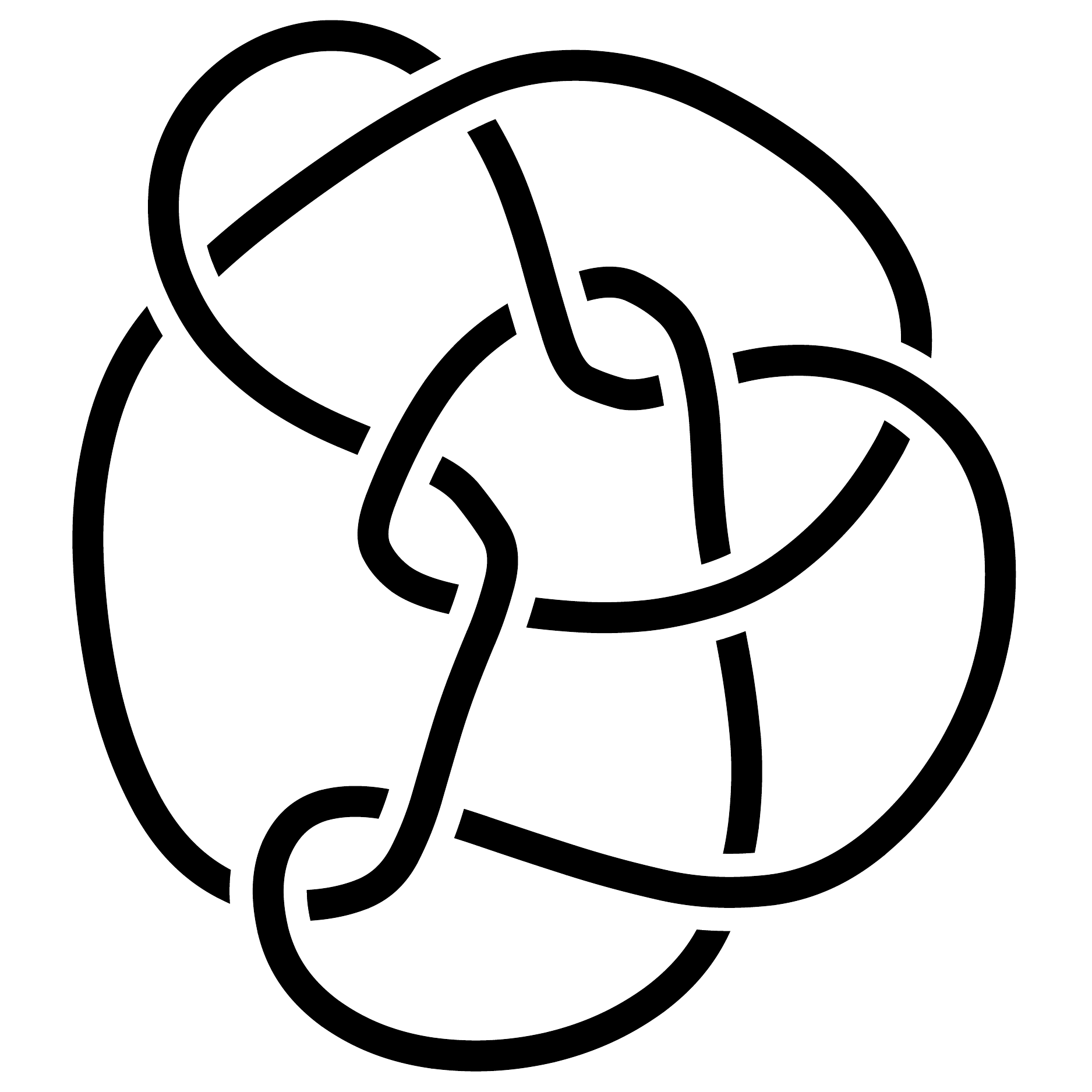}}&   \\
\nopagebreak   & $11n_{142}$ & $1-8t+15t^2-8t^3+t^4$ \\
\nopagebreak  &  &  \\
\multirow{3}{*}{\includegraphics[height=20mm]{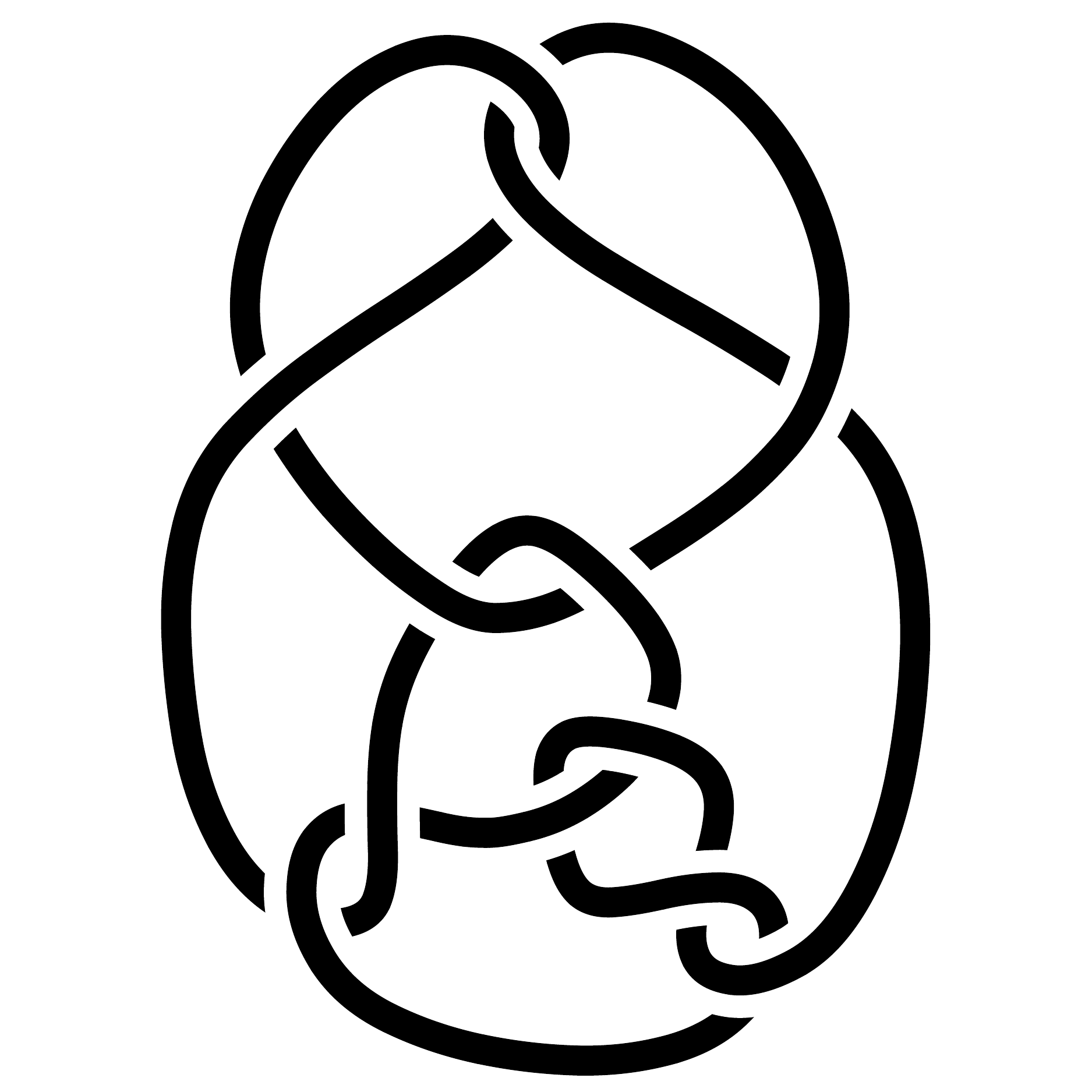}}&  \\
\nopagebreak   & $12a_{0125}$ & $1-12t+44t^2-67t^3+44t^4-12t^5+t^6$ \\
\nopagebreak  &  &  \\
\multirow{3}{*}{\includegraphics[height=20mm]{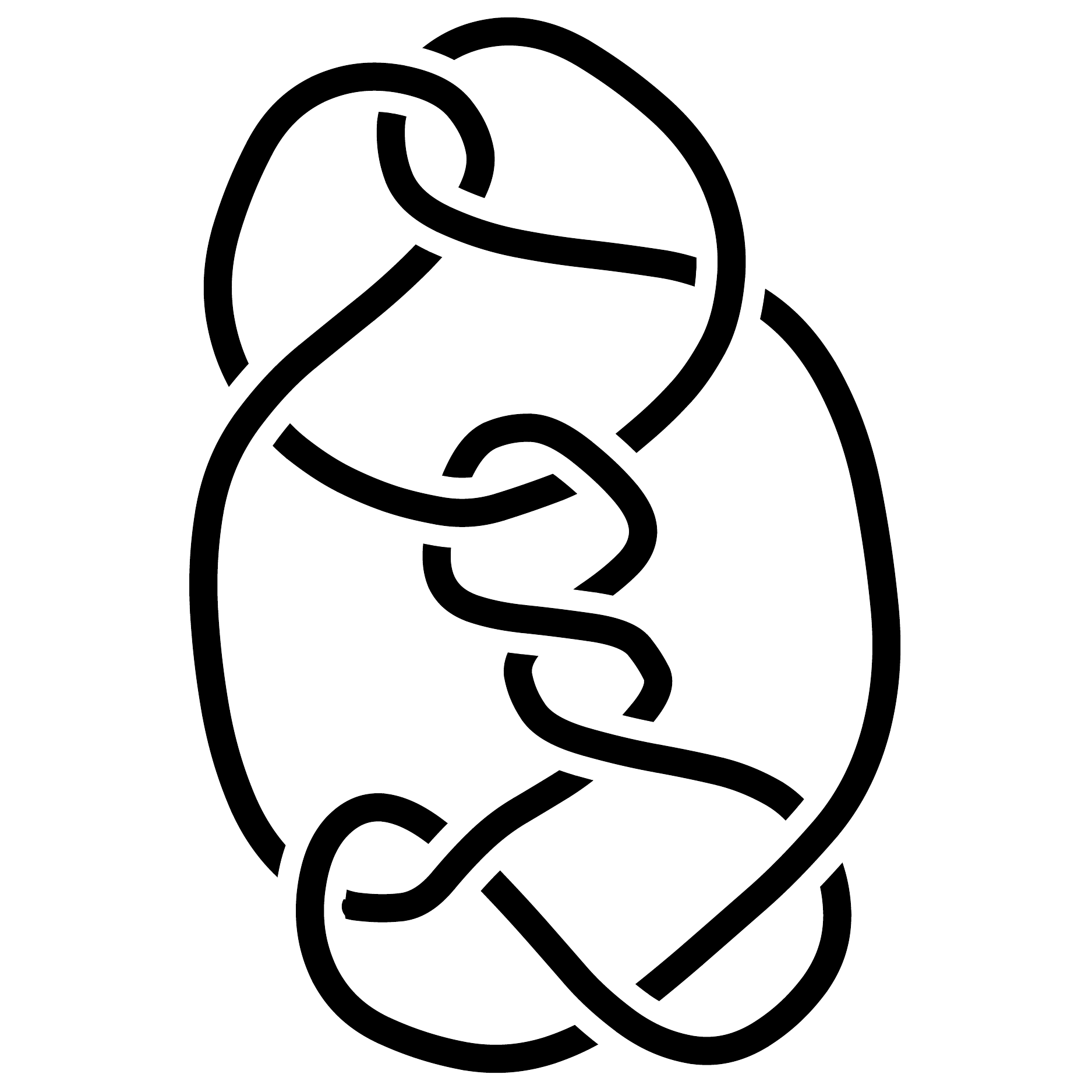}}&   \\
\nopagebreak   & $12a_{0181}$ & $1-11t+40t^2-61t^3+40t^4-11t^5+t^6$ \\
\nopagebreak  &  &  \\
\multirow{3}{*}{\includegraphics[height=20mm]{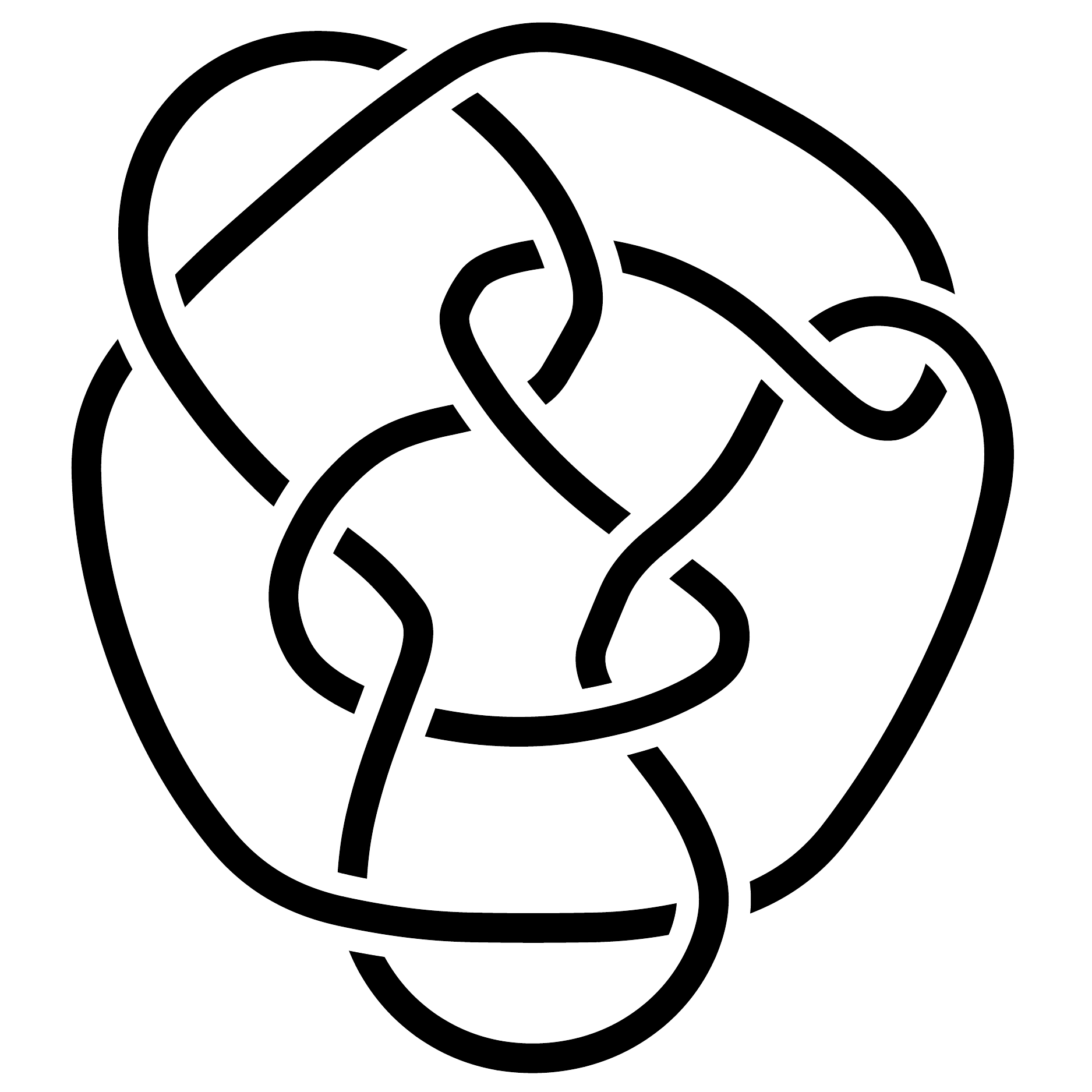}}&   \\
\nopagebreak   & $12a_{1124}$ & $1-13t+50t^2-77t^3+50t^4-13t^5+t^6$ \\
\nopagebreak  &  & \\
\multirow{3}{*}{\includegraphics[height=20mm]{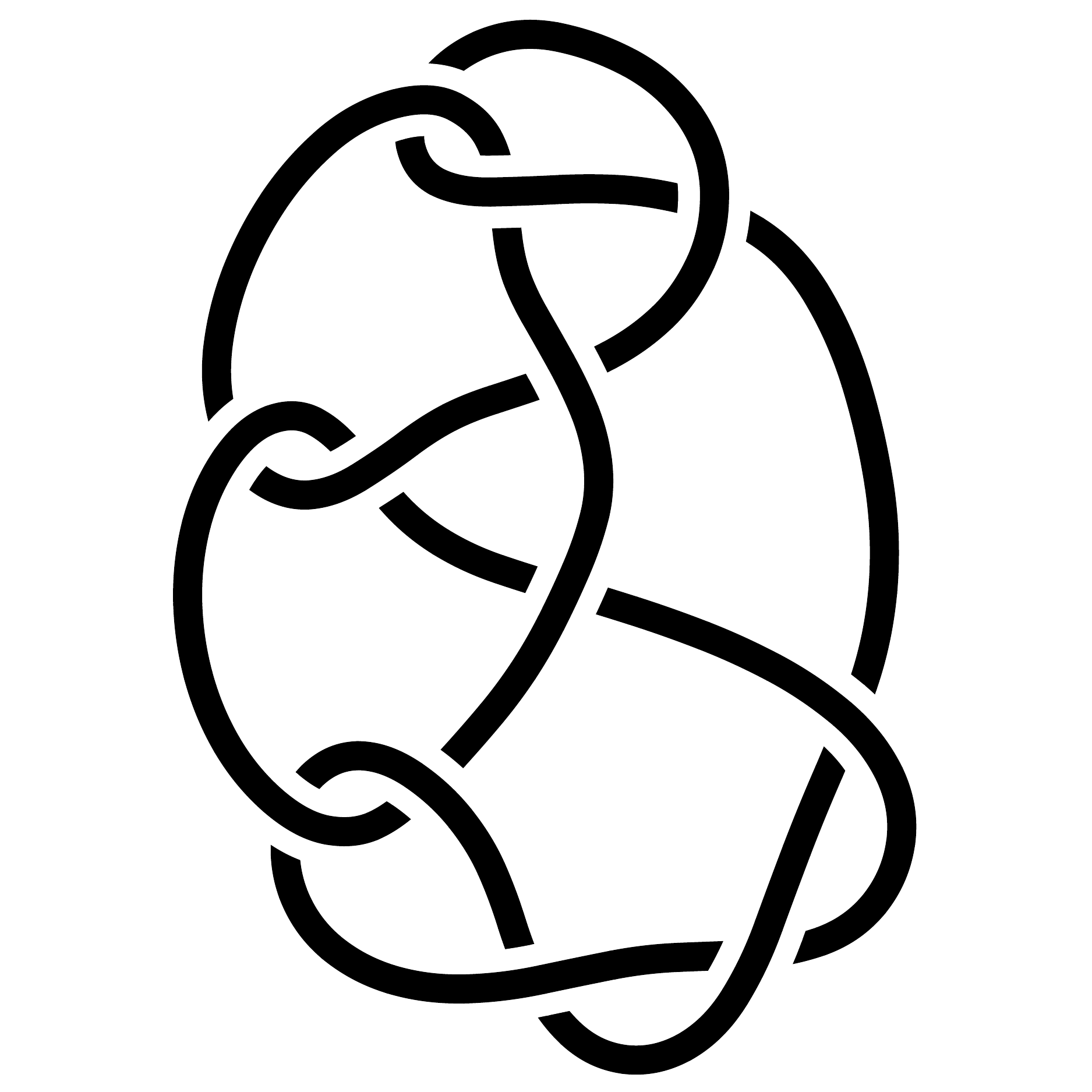}}&   \\
\nopagebreak   & $12n_{0013}$ &  $1-7t+13t^2-7t^3+t^4$\\
\nopagebreak  &  &  \\
\multirow{3}{*}{\includegraphics[height=20mm]{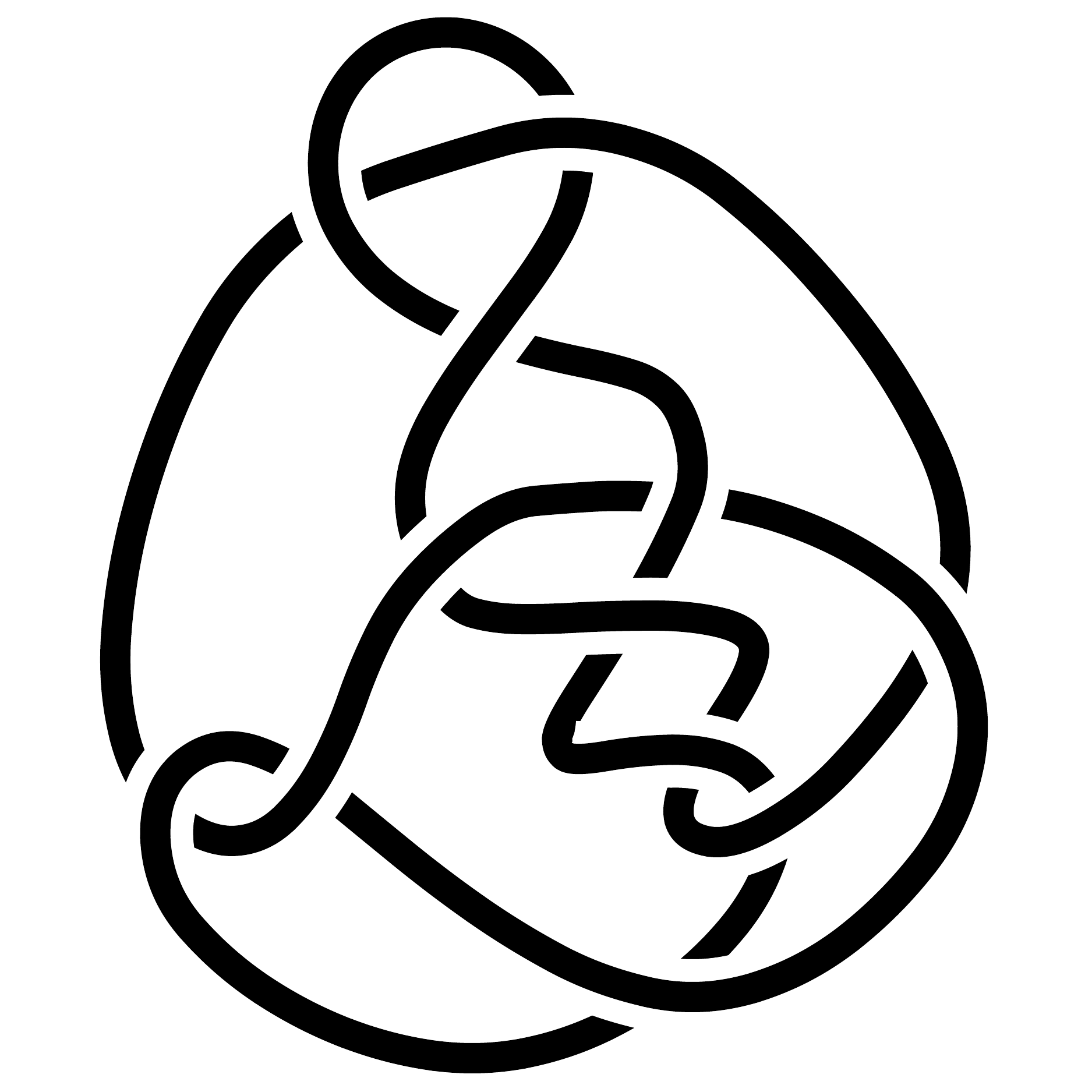}}&   \\
\nopagebreak   & $12n_{0145}$ & $1-6t+11t^2-6t^3+t^4$ \\
\nopagebreak  &  &  \\
\multirow{3}{*}{\includegraphics[height=20mm]{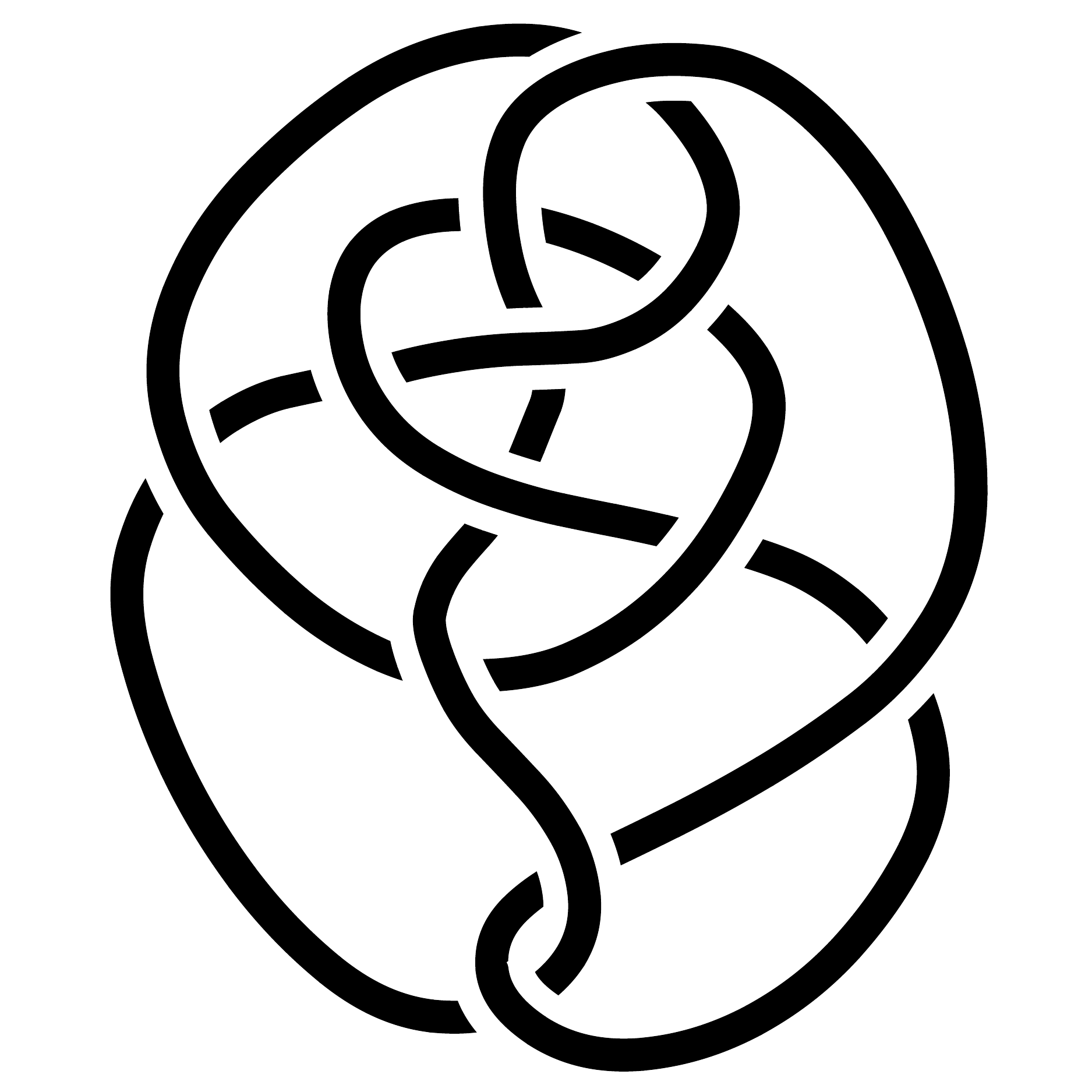}}&    \\
\nopagebreak   & $12n_{0462}$ & $1-6t+11t^2-6t^3+t^4$ \\
\nopagebreak  &  &  \\
\multirow{3}{*}{\includegraphics[height=20mm]{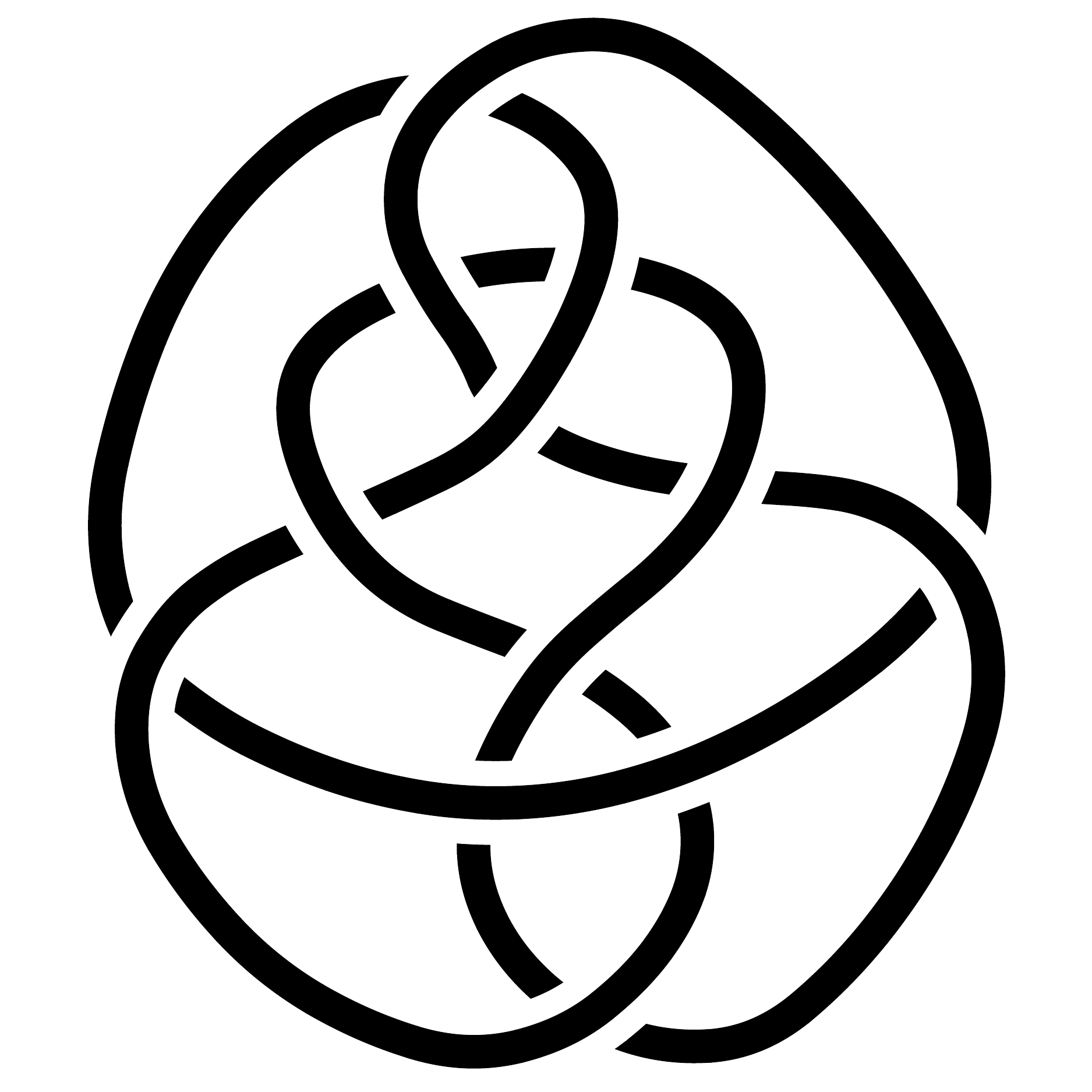}}&    \\
\nopagebreak   & $12n_{0838}$ & $1-6t+11t^2-6t^3+t^4$\\
\nopagebreak  &  &  \\

\end{longtable}
}

The prime knots with $12$ or fewer crossings which are known to have non bi-orderable group, because they are fibred and have Alexander polynomials without positive real roots, are as follows:

$3_{1}$, $5_{1}$, $6_{3}$, $7_{1}$, $7_{7}$, $8_{7}$, $8_{10}$, $8_{16}$, $8_{19}$, $8_{20}$, $9_{1}$, $9_{17}$, $9_{22}$, $9_{26}$, $9_{28}$, $9_{29}$, $9_{31}$, $9_{32}$, $9_{44}$, $9_{47}$, $10_{5}$, $10_{17}$, $10_{44}$, $10_{47}$, $10_{48}$, $10_{62}$, $10_{69}$, $10_{73}$, $10_{79}$, $10_{85}$, $10_{89}$, $10_{91}$, $10_{99}$, $10_{100}$, $10_{104}$, $10_{109}$, $10_{118}$, $10_{124}$, $10_{125}$, $10_{126}$, $10_{132}$, $10_{139}$, $10_{140}$, $10_{143}$, $10_{145}$, $10_{148}$, $10_{151}$, $10_{152}$, $10_{153}$, $10_{154}$, $10_{156}$, $10_{159}$, $10_{161}$, $10_{163}$, $11a_{9}$,  $11a_{14}$, $11a_{22}$,  $11a_{24}$, $11a_{26}$, $11a_{35}$, $11a_{40}$, $11a_{44}$, $11a_{47}$, $11a_{53}$, $11a_{72}$, $11a_{73}$, $11a_{74}$, $11a_{76}$, $11a_{80}$, $11a_{83}$, $11a_{88}$, $11a_{106}$, $11a_{109}$, $11a_{113}$, $11a_{121}$, $11a_{126}$, $11a_{127}$, $11a_{129}$, $11a_{160}$, $11a_{170}$, $11a_{175}$, $11a_{177}$, $11a_{179}$, $11a_{180}$, $11a_{182}$, $11a_{189}$, $11a_{194}$, $11a_{215}$, $11a_{233}$, $11a_{250}$, $11a_{251}$, $11a_{253}$, $11a_{257}$, $11a_{261}$, $11a_{266}$, $11a_{274}$, $11a_{287}$, $11a_{288}$, $11a_{289}$, $11a_{293}$, $11a_{300}$, $11a_{302}$, $11a_{306}$, $11a_{315}$, $11a_{316}$, $11a_{326}$, $11a_{330}$, $11a_{332}$, $11a_{346}$, $11a_{367}$,  $11n_{7}$,  $11n_{11}$,  $11n_{12}$, $11n_{15}$, $11n_{22}$, $11n_{23}$, $11n_{24}$, $11n_{25}$, $11n_{28}$, $11n_{41}$, $11n_{47}$, $11n_{52}$, $11n_{54}$, $11n_{56}$, $11n_{58}$, $11n_{61}$, $11n_{74}$, $11n_{76}$, $11n_{77}$, $11n_{78}$, $11n_{82}$, $11n_{87}$, $11n_{92}$, $11n_{96}$, $11n_{106}$, $11n_{107}$, $11n_{112}$, $11n_{124}$, $11n_{125}$, $11n_{127}$, $11n_{128}$, $11n_{129}$, $11n_{131}$, $11n_{133}$, $11n_{145}$, $11n_{146}$, $11n_{147}$, $11n_{149}$, $11n_{153}$, $11n_{154}$, $11n_{158}$, $11n_{159}$, $11n_{160}$, $11n_{167}$, $11n_{168}$, $11n_{173}$, $11n_{176}$, $11n_{182}$, $11n_{183}$, $12a_{0001}$, $12a_{0008}$, $12a_{0011}$, $12a_{0013}$, $12a_{0015}$, $12a_{0016}$, $12a_{0020}$, $12a_{0024}$, $12a_{0026}$, $12a_{0030}$, $12a_{0033}$, $12a_{0048}$, $12a_{0058}$, $12a_{0060}$, $12a_{0066}$, $12a_{0070}$, $12a_{0077}$, $12a_{0079}$, $12a_{0080}$, $12a_{0091}$, $12a_{0099}$, $12a_{0101}$, $12a_{0111}$, $12a_{0115}$, $12a_{0119}$, $12a_{0134}$, $12a_{0139}$, $12a_{0141}$, $12a_{0142}$, $12a_{0146}$, $12a_{0157}$, $12a_{0184}$, $12a_{0186}$, $12a_{0188}$, $12a_{0190}$, $12a_{0209}$, $12a_{0214}$, $12a_{0217}$, $12a_{0219}$, $12a_{0222}$, $12a_{0245}$, $12a_{0246}$, $12a_{0250}$, $12a_{0261}$, $12a_{0265}$, $12a_{0268}$, $12a_{0271}$, $12a_{0281}$, $12a_{0299}$, $12a_{0316}$, $12a_{0323}$, $12a_{0331}$, $12a_{0333}$, $12a_{0334}$, $12a_{0349}$, $12a_{0351}$, $12a_{0362}$, $12a_{0363}$, $12a_{0369}$, $12a_{0374}$, $12a_{0386}$, $12a_{0396}$, $12a_{0398}$, $12a_{0426}$, $12a_{0439}$, $12a_{0452}$, $12a_{0464}$, $12a_{0466}$, $12a_{0469}$, $12a_{0473}$, $12a_{0476}$, $12a_{0477}$, $12a_{0479}$, $12a_{0497}$, $12a_{0499}$, $12a_{0515}$, $12a_{0536}$, $12a_{0561}$, $12a_{0565}$, $12a_{0569}$, $12a_{0576}$, $12a_{0579}$, $12a_{0629}$, $12a_{0662}$, $12a_{0696}$, $12a_{0697}$, $12a_{0699}$, $12a_{0700}$, $12a_{0706}$, $12a_{0707}$, $12a_{0716}$, $12a_{0815}$, $12a_{0824}$, $12a_{0835}$, $12a_{0859}$, $12a_{0864}$, $12a_{0867}$, $12a_{0878}$, $12a_{0898}$, $12a_{0916}$, $12a_{0928}$, $12a_{0935}$, $12a_{0981}$, $12a_{0984}$, $12a_{0999}$, $12a_{1002}$, $12a_{1013}$, $12a_{1027}$, $12a_{1047}$, $12a_{1065}$, $12a_{1076}$, $12a_{1105}$, $12a_{1114}$, $12a_{1120}$, $12a_{1122}$, $12a_{1128}$, $12a_{1168}$, $12a_{1176}$, $12a_{1188}$, $12a_{1203}$, $12a_{1219}$, $12a_{1220}$, $12a_{1221}$, $12a_{1226}$, $12a_{1227}$, $12a_{1230}$, $12a_{1238}$, $12a_{1246}$, $12a_{1248}$, $12a_{1253}$, $12n_{0005}$, $12n_{0006}$, $12n_{0007}$, $12n_{0010}$, $12n_{0016}$, $12n_{0019}$, $12n_{0020}$, $12n_{0038}$, $12n_{0041}$, $12n_{0042}$, $12n_{0052}$, $12n_{0064}$, $12n_{0070}$, $12n_{0073}$, $12n_{0090}$, $12n_{0091}$, $12n_{0092}$, $12n_{0098}$, $12n_{0104}$, $12n_{0105}$, $12n_{0106}$, $12n_{0113}$, $12n_{0115}$, $12n_{0120}$, $12n_{0121}$, $12n_{0125}$, $12n_{0135}$, $12n_{0136}$, $12n_{0137}$, $12n_{0139}$, $12n_{0142}$, $12n_{0148}$, $12n_{0150}$, $12n_{0151}$, $12n_{0156}$, $12n_{0157}$, $12n_{0165}$, $12n_{0174}$, $12n_{0175}$, $12n_{0184}$, $12n_{0186}$, $12n_{0187}$, $12n_{0188}$, $12n_{0190}$, $12n_{0192}$, $12n_{0198}$, $12n_{0199}$, $12n_{0205}$, $12n_{0226}$, $12n_{0230}$, $12n_{0233}$, $12n_{0235}$, $12n_{0242}$, $12n_{0261}$, $12n_{0272}$, $12n_{0276}$, $12n_{0282}$, $12n_{0285}$, $12n_{0296}$, $12n_{0309}$, $12n_{0318}$, $12n_{0326}$, $12n_{0327}$, $12n_{0328}$, $12n_{0329}$, $12n_{0344}$, $12n_{0346}$, $12n_{0347}$, $12n_{0348}$, $12n_{0350}$, $12n_{0352}$, $12n_{0354}$, $12n_{0355}$, $12n_{0362}$, $12n_{0366}$, $12n_{0371}$, $12n_{0372}$, $12n_{0377}$, $12n_{0390}$, $12n_{0392}$, $12n_{0401}$, $12n_{0402}$, $12n_{0405}$, $12n_{0409}$, $12n_{0416}$, $12n_{0417}$, $12n_{0423}$, $12n_{0425}$, $12n_{0426}$, $12n_{0427}$, $12n_{0437}$, $12n_{0439}$, $12n_{0449}$, $12n_{0451}$, $12n_{0454}$, $12n_{0456}$, $12n_{0458}$, $12n_{0459}$, $12n_{0460}$, $12n_{0466}$, $12n_{0468}$, $12n_{0472}$, $12n_{0475}$, $12n_{0484}$, $12n_{0488}$, $12n_{0495}$, $12n_{0505}$, $12n_{0506}$, $12n_{0508}$, $12n_{0514}$, $12n_{0517}$, $12n_{0518}$, $12n_{0522}$, $12n_{0526}$, $12n_{0528}$, $12n_{0531}$, $12n_{0538}$, $12n_{0543}$, $12n_{0549}$, $12n_{0555}$, $12n_{0558}$, $12n_{0570}$, $12n_{0574}$, $12n_{0577}$, $12n_{0579}$, $12n_{0582}$, $12n_{0591}$, $12n_{0592}$, $12n_{0598}$, $12n_{0601}$, $12n_{0604}$, $12n_{0609}$, $12n_{0610}$, $12n_{0613}$, $12n_{0619}$, $12n_{0621}$, $12n_{0623}$, $12n_{0627}$, $12n_{0629}$, $12n_{0634}$, $12n_{0640}$, $12n_{0641}$, $12n_{0642}$, $12n_{0647}$, $12n_{0649}$, $12n_{0657}$, $12n_{0658}$, $12n_{0660}$, $12n_{0666}$, $12n_{0668}$, $12n_{0670}$, $12n_{0672}$, $12n_{0673}$, $12n_{0675}$, $12n_{0679}$, $12n_{0681}$, $12n_{0683}$, $12n_{0684}$, $12n_{0686}$, $12n_{0688}$, $12n_{0690}$, $12n_{0694}$, $12n_{0695}$, $12n_{0697}$, $12n_{0703}$, $12n_{0707}$, $12n_{0708}$, $12n_{0709}$, $12n_{0711}$, $12n_{0717}$, $12n_{0719}$, $12n_{0721}$, $12n_{0725}$, $12n_{0730}$, $12n_{0739}$, $12n_{0747}$, $12n_{0749}$, $12n_{0751}$, $12n_{0754}$, $12n_{0761}$, $12n_{0762}$, $12n_{0781}$, $12n_{0790}$, $12n_{0791}$, $12n_{0798}$, $12n_{0802}$, $12n_{0803}$, $12n_{0835}$, $12n_{0837}$, $12n_{0839}$, $12n_{0842}$, $12n_{0848}$, $12n_{0850}$, $12n_{0852}$, $12n_{0866}$, $12n_{0871}$, $12n_{0887}$, $12n_{0888}$.

\bibliographystyle{plain}

\bibliography{eigenvalues}

\end{document}